\documentclass[11pt, reqno]{amsart}
\usepackage[T1]{fontenc}
\usepackage[english]{babel}

\usepackage{amsmath}
\usepackage{amssymb}
\usepackage{amsthm}
\usepackage{bbm}
\usepackage{mathrsfs}
\usepackage{epsfig}
\usepackage{enumerate}
\usepackage{mathtools}
\usepackage{graphicx}
\usepackage[usenames,dvipsnames]{xcolor}
\usepackage[colorlinks=true,linkcolor=NavyBlue,urlcolor=RoyalBlue,citecolor=PineGreen,bookmarks=true,bookmarksopen=true,bookmarksopenlevel=2,unicode=true,linktocpage]
{hyperref}

\newtheorem{theorem}{Theorem}
\newtheorem{lemma}[theorem]{Lemma}

\newtheorem{conjecture}[theorem]{Conjecture}

\newcommand{\D}[2]{ D_{#1}(#2)}

\def\d{\;{\rm d}}
\def\D{\mathcal{D}}
\def\R{\mathbb{R}}
\def\1{\mathbbm{1}}
\def\P{\mathbb{P}}

\def\epsilon{\varepsilon}

\def\build#1_#2^#3{\mathrel{\mathop{\kern 0pt#1}\limits_{#2}^{#3}}}

\title[Planar Brownian motion winds evenly along its trajectory]{Planar Brownian motion winds evenly \\along its trajectory}

\author{Isao Sauzedde}
\address{Isao Sauzedde -- LPSM, Sorbonne Universit\'e, Paris}
\email{isao.sauzedde@lpsm.paris}


\begin{document}

\begin{abstract}
Let $\mathcal{D}_N$ be the set of points around which a planar Brownian motion winds at least~$N$ times. We prove that the random measure on the plane with density $2 \pi N\mathbbm{1}_{\D_N}$ with respect to the Lebesgue measure
converges almost surely weakly, as $N$ tends to infinity, towards the occupation measure of the Brownian motion.
\end{abstract}
\maketitle
\section{Introduction}
Let $X:[0,1]\to \mathbb{R}^2$ be a planar Brownian motion started from $0$.
Let $\bar{X}$ be the oriented loop obtained by concatenating $X$ with the straight line segment joining $X_1$ to $X_0$.

For each point $z$ in $\R^2$ outside the range of $\bar{X}$, let $\theta(z)$ be the number of times $\bar{X}$ winds around~$z$. For $z$ on the range of $\bar X$, we set $\theta(z)=0$. Define
\[\D_N=\{z\in \R^2 : \theta(z)\geq N\}.\]
The Lebesgue measure $|\D_N|$ of this set is known to be of the order of $\frac{1}{2\pi N}$. More precisely, Werner proved in \cite{werner} that the following convergence holds:
\begin{equation}
\label{eq:werner}
2\pi N |\D_N| \overset{L^2}{\underset{N\to \infty}\longrightarrow} 1.
\end{equation}
For all $N\geq 1$, we denote by $\mu_{N}$ the random measure on the plane with density $2 \pi N\mathbbm{1}_{\D_N}$ with respect to the Lebesgue measure:
\[\d\mu_N(z)=2 \pi N\mathbbm{1}_{\D_N}(z)\d z.\]

Let $\nu$ be the occupation measure of $X$, defined as the push-forward of the Lebesgue measure on $[0,1]$ by $X$. In other words, $\nu$ is the random Borel probability measure  on the plane characterised by the fact that for every continuous test function $f:\R^2\to \R$,
\[\int_{\R^2} f \d\nu= \int_0^1 f(X_t)\d t.\]

The main result of this paper is the following.


\begin{theorem}
\label{th:main}
Almost surely, $\mu_N \underset{N\to \infty}\Longrightarrow \nu$.
\end{theorem}

To be clear, we mean that almost surely, for all bounded continuous function $f:\R^2\to \R$, the following convergence holds:
\[\lim_{N\to \infty} 2\pi N \int_{\R^2} f(z) \mathbbm{1}_{[N,+\infty)}(\theta(z))\d z =\int_{0}^{1} f(X_u) \d u.
\]
The assumption that the test function is bounded is not essential, because almost surely, the supports of the measures $\mu_N$, $N\geq 1$ and $\nu$ are contained in the convex hull of the range of $X$, which is compact.

In the course of the proof, we will obtain an estimation of the rate of convergence in terms of the modulus of continuity of the test function $f$ (see Lemma \ref{le:main}).
\medskip

The study of the windings of the planar Brownian motion has a long history. The first investigations were mostly concerned with the winding around a fixed point, the most prominent example being the celebrated Spitzer theorem \cite{spitzer}. There followed among other works a computation by Yor of the exact law of the winding \cite{mansuyYor,yor}, as well as many fine asymptotic results concerning related functionals (see for example \cite{zhan} and references therein).


In \cite{werner,werner2}, Werner shifted the attention from the winding around a point to the winding as a \emph{function}, as well as to the set of points with a given winding number. He established, for instance, in \cite{werner}, the convergence \eqref{eq:werner}. His results suggest in particular that when $N$ is large, the set $\D_N$, which is located near the trajectory $X$, has a very balanced distribution along this trajectory. Our main result gives a rigorous statement of this idea.

Our proof uses some results that we obtained in our previous work \cite{LAWA} on this subject, and which we recall briefly in the next section for the convenience of the reader.

\section{Prior results}

The Brownian motion $X$ is defined under a probability that we denote by $\P$.

Let $T$ be a positive integer. For all $i\in \{1,\ldots,T\}$, let $X^i$ be the restriction of $X$ to the interval $[\tfrac{i-1}{T}, \frac{i}{T}]$. As we did for $X$, let us denote by $\bar{X}^i$ the concatenation of $X^i$ with a straight line segment from $X_{\frac{i}{T}}$ to $X_{\frac{i-1}{T}}$, and by $\theta^i$ the winding function of the loop $\bar{X}^i$, taken to be $0$ on the trajectory. We then set, for all $N\geq 1$,
\[\D^i_N =\{z \in \R^2 : \theta^{i}(z)\geq N\} \ \text{ and }\
\D^{i,j}_N =
\{z\in \R^2 :
|\theta^i(z)|\geq N, |\theta^j(z)|\geq N\},
\]
with absolute values intended in the second definition.

Our proof of Theorem \ref{th:main} relies on the following lemmas, which are mild reformulations of results that we proved in \cite{LAWA} (see Equation (28), Theorem 1.5 and Lemma 2.4 there).
\begin{lemma}
\label{le:decompo}
Let $\mu$ be a Borel measure on $\R^2$, absolutely continuous with respect to the Lebesgue measure.
For all positive integers $N,T,M$ such that $T(M+1)<N$,
\[
\sum_{i=1}^T \mu\big(\D^i_{N+T+M(T-1)}\big)  - \hspace{-0.2cm}\sum_{1\leq i<j\leq T} \hspace{-0.2cm} \mu\big(\D^{i,j}_M\big)
\leq \mu(\D_N)
\leq
\sum_{i=1}^T \mu\big(\D^i_{N-T-M(T-1)}\big)  + \hspace{-0.2cm}\sum_{1\leq i<j\leq T} \hspace{-0.2cm}\mu\big(\D^{i,j}_M\big).
\]
\end{lemma}

\begin{lemma}
\label{lemma:psmax1}
For all $\delta< \frac{1}{2}$ and $p>0$, there exists 
$C>0$ such that for all $N\geq 1$ and all $R>0$,
\[
\mathbb{P}\Big( N^\delta \big| 2\pi N|\D_N|-1 \big| \geq R
\Big)\leq C R^{-p}.
\]
\end{lemma}

\begin{lemma}
\label{lemma:boundedSquare}
For all $\epsilon>0$, there exists $C>0$ such that for all positive integers $T,M$,
\[
\mathbb{E}\Big[ \Big(\hspace{-0cm}\sum_{1\leq i<j\leq T} \hspace{-0cm} |\D^{i,j}_M| \Big)^2 \Big]\leq C M^{-4+\epsilon} T^{1+\epsilon}.
\]
\end{lemma}

\section{Proof of the theorem}

Let $f:\R^2\to \R$ be a bounded continuous function. Let $\omega_f$ be the modulus of continuity of $f$: for all $t\geq 0$,
\[\omega_f(t)=\sup \{ |f(z)-f(w)|: z,w\in \R^2, \|z-w\|\leq t\} \in [0,+\infty].\]
For all Borel subset $E$ of $\R^2$, we also set $f(E)=\int_E f(z)\d z$.

For $\alpha\in (0,\tfrac{1}{2})$, let $\|X\|_{\mathcal{C}^\alpha}$ denote the $\alpha$-H\"older norm of the Brownian motion:
\[
\|X\|_{\mathcal{C}^\alpha}=\sup_{0\leq s<t\leq 1}  \frac{\|X_t-X_s\|}{|t-s|^\alpha}.
\]

We have the following quantitative estimation.
\begin{lemma}
\label{le:main}
For all $t\in (0,\tfrac{2}{5})$ and $\alpha\in(0,\tfrac{1}{2})$, there exists $\eta>0$ such that
$\P$-almost surely, there
  exists a constant $C$ such that for all bounded continuous function $f:\R^2\to \R$ and all $N\geq 1$,
  \[
  \left|2\pi N f(\mathcal{D}_N)-  \int_0^1 f(X_u)\d u \right|
  \leq C \big( \omega_f(2 \|X\|_{\mathcal{C}^\alpha} N^{-\alpha t}  )
  +\|f\|_\infty N^{-\eta} \big).\]
\end{lemma}
Let us explain why this lemma directly implies Theorem \ref{th:main}.
\begin{proof}[Proof of Theorem {\ref{th:main}} assuming Lemma {\ref{le:main}}]
Thanks to the Portmanteau theorem, is suffices to show that $\P$-almost surely, for any bounded Lipschitz continuous function $f$,
\[  \left|2\pi N f(\mathcal{D}_N)-  \int_0^1 f(X_u)\d u \right|\underset{N\to +\infty}{\longrightarrow} 0. \]
For such a function $f$, one has $\omega_f(t)\leq \|f\|_{\rm Lip}\, t $ and the result follows from Lemma \ref{le:main} applied for instance to $t=\tfrac{1}{5}$ and $\alpha=\tfrac{1}{4}$. \end{proof}

In order to prove Lemma \ref{le:main}, we introduce the following subset of $\mathbb{N}$, which depends on a positive real parameter $\gamma>1$:
\[\mathbb{N}^\gamma =\{ \lfloor K^{\gamma }\rfloor: K\in \mathbb{N} \}\setminus\{0\}.\]
Let us fix two positive real parameters $t$ and $m$ with $m+t<1$ and set, for all $N\geq 1$, $T=\lfloor N^t\rfloor$ and $M=\lfloor N^m \rfloor$. We advise the reader to think of $m$ as being larger than $\frac{1}{2}$, and of $t$ as a small number. Precise conditions can be found in the statement of Lemma \ref{le:casXi}.

We also set $N'= \max \{ n\in \mathbb{N}^{\gamma}: n\leq N-T-M(T-1)\}$, which is well defined when $N$ is large enough. The difference between $N$ and $N'$ is
$O(N^{1-1/\gamma}+ N^{m+t})$.

We also define the following events, which depend on $t$ and $m$, and also on other positive real parameters $s,\zeta,\delta$:
\begin{align*}
E_N&= \big\{  \forall i \in \{1,\dots, T\}, \ 
{N'}^\delta  \left|2\pi N'| \D^{i}_{N'}|    -\tfrac{1}{T}  \right|  \leq T^{-\frac{1}{2}+ \frac{s}{t}}  \big\},\\
F_N&= \Big\{  
 \hspace{-0cm}\sum_{1\leq i<j \leq T}\hspace{-0cm} |\D^{i,j}_M|\leq N^{-1-\zeta} \Big\},\\
 G_N&= \big\{  \forall i \in \{1,\dots, T\},\ 
 2\pi N |\D^{i}_{N'}|\leq \tfrac{2}{T} \big\}.
\end{align*}


The proof goes in three steps. In the first (Lemma \ref{le:events}), we show that with an appropriate choice of $\gamma$, almost surely, the events $E_N$, $F_N$ and $G_N$ are realised for all $N\in {\mathbb N}^\gamma$ large enough. In a second step (Lemma \ref{le:casXi}), we show that on this almost sure event, for every bounded continuous function, and for all $N\in {\mathbb N}^\gamma$, the conclusion of Lemma \ref{le:main} holds. In the third step, we show that the conclusion holds not only for $N\in {\mathbb N}^\gamma$, but for all $N\in \mathbb N$.

Let us collect in one place the assumptions that we make on the parameters that we introduced. These assumptions are organised in such a way that if enforced in the natural reading order, they are always satisfiable.

\begin{align}\tag{A}\label{eq:hyp}
0<\alpha<\tfrac{1}{2}\ ,\ 0<t<\tfrac{2}{5}\ ,
\begin{array}{c}
\tfrac{1}{2}+\tfrac{t}{4}<m< 1-t \ , \
 0<\zeta<2m-1-\tfrac{t}{2},\\[5pt]
0<s<\tfrac{1}{2}-\tfrac{t}{2} \ , \ \tfrac{t}{2}+s< \delta<\tfrac{1}{2},
\end{array} \
\gamma >\max \big(\tfrac{1}{2s}, \tfrac{1}{4m-t-2-2\zeta}\big).
\end{align}
From now on, we always assume that these assumptions are satisfied.

\begin{lemma}
  \label{le:events}
The event
  $\displaystyle \bigcup_{N_0\geq 1} \bigcap_{\substack{ N\in \mathbb{N}^{\gamma} \\ N\geq N_0 }} (E_N\cap F_N \cap G_N)$ has probability $1$.

\end{lemma}
\begin{proof}
  The scaling properties of the Brownian motion imply that $|\mathcal{D}^i_{N'}|$ is equal in distribution to $T^{-1}|\mathcal{D}_{N'}|$. Thus,
  \[ 1- \mathbb{P}(E_N)\leq T \mathbb{P}( {N'}^\delta \big|2 \pi N'|\mathcal{D}_{N'}|-1\big|\geq T^{\frac{1}{2}+\frac{s}{t}}).\]
  Using Lemma \ref{lemma:psmax1} with $p=2$ gives
  \[ 1- \mathbb{P}(E_N)\leq C T^{-\frac{2s}{t}},\]
  and for $N$ large enough, this quantity is smaller than $2 C N^{-2s}$.
    In particular,
    \[  \sum_{N\in \mathbb{N}^\gamma} \big(1-\P(E_N)\big)\leq 2 C \sum_{K=1}^{+\infty} K^{-2s \gamma}.
  \]

  Besides, by Markov inequality,
  \[ 1- \mathbb{P}(F_N)\leq N^{2+2\zeta}\,  \mathbb{E}\Big[ \Big(\sum_{1\leq i<j \leq T} |\mathcal{D}^{i,j}_M| \Big)^2 \Big].\]
  By Lemma \ref{lemma:boundedSquare}, for any $\epsilon>0$, there exists $C$ such that for all $N$,
  \[ 1- \mathbb{P}(F_N)\leq C N^{-4m+t+2+2\zeta+\epsilon  }.\]
  In particular,
\[  \sum_{N\in \mathbb{N}^\gamma} \big(1-\P(F_N)\big)\leq C \sum_{K=1}^{+\infty} K^{\gamma(-4m+t+2+2\zeta+\epsilon)  }.
  \]
  We assumed that $\gamma>\frac{1}{4m-t-2-2\zeta}$, so that there exists $\epsilon>0$ such that
  $\gamma>\frac{1}{4m-t-2-2\zeta-\epsilon}$. Since we also assumed that $\gamma>\frac{1}{2s}$,
the series
  \[ \sum_{K=1}^{+\infty} K^{-\gamma(4m-t-2-2\zeta-\epsilon) } \qquad \mbox{and}\qquad \sum_{K=1}^{+\infty} K^{-\gamma(2s) }\]
  are both convergent.

  Using Borel--Cantelli lemma, we conclude the proof, but for the presence of $G_N$. However, using the fact that $N'$ is not larger than $N$ and equivalent to $N$ as $N$ tends to infinity, and the inequality $T\leq N^t$, one verifies that if $t+2s<2\delta$, then for $N$ large enough, the inclusion $E_N\subset G_N$ holds. Hence, the proof is complete.
\end{proof}


We now turn to the second step of the proof.

\begin{lemma}
\label{le:casXi}
Almost surely, there
  exists a constant $C$ such that for all $N\in \mathbb{N}^\gamma $ and all bounded continuous function $f:\R^2\to \R$,
  \[
  \left|2\pi N f(\mathcal{D}_N)\! - \!\! \int_0^1\! \! f(X_u)\d u \right|
  \leq\! C \Big( \omega_f\big(\|X\|_{\mathcal{C}^\alpha} T^{-\alpha}\big)
  +\|f\|_\infty ( N^{-1+m+t} +N^{-\frac{1}{\gamma}+1}+ N^{-\delta+\frac{t}{2}+s}+N^{-\zeta})\Big).\]
\end{lemma}
\begin{proof}
We first assume that $f$ is non-negative.
Replacing $C$ if necessary by a larger constant, it suffices to show the inequality for $N\geq N_0$, for a possibly random $N_0$ which does not depend on~$f$. Using Lemma \ref{le:events}, we can thus assume that the event $E_N\cap F_N\cap G_{N}$ holds.

Using Lemma \ref{le:decompo}, the assumption that $f$ is non-negative and the fact that the sequence $(\D^i_N)_{N\geq 1}$ is non-increasing, we have
\begin{align}
  N f({\D}_N)
  &\leq \sum_{i=1}^T N f({\D}^{i}_{N-T-M(T-1) }) +\sum_{1\leq i<j\leq T} N f({\D}^{i,j}_M)\nonumber\\
&\leq\sum_{i=1}^T N f({\D}^{i}_{N'}) +\sum_{1\leq i<j\leq T} N f({\D}^{i,j}_M).
  \label{eq:decompoSmall}
\end{align}
Besides, ${\D}^{i}_{N'}$ is contained in the convex hull of the trajectory of $X$ between the times $\tfrac{i}{T}$ and $\tfrac{i+1}{T}$, hence in the ball of center $X_{\frac{i}{T}}$ and radius $\|X\|_{\mathcal{C}^\alpha} T^{-\alpha}$, so that
\[
N f({\D}^{i}_{N'})
\leq N |{\D}^{i}_{N'}| f(X_{\frac{i}{T} } )+ N |{\D}^{i}_{N'}|\omega_f(\|X\|_{\mathcal{C}^\alpha} T^{-\alpha} ).
\]
We replace in \eqref{eq:decompoSmall} and force the apparition of a Riemann sum by decomposing $N |{\D}^{i}_{N'}|$ into
\[\frac{1}{2\pi T}+ \frac{N-N' }{2\pi T N' } +
 N \left(
|\D^{i}_{N'}|-\tfrac{1}{2\pi T N'}\right).\]
We obtain
\begin{align*}
\sum_{i=1}^T N f({\D}^{i}_{N'})
& \leq
\sum_{i=1}^T \tfrac{1}{2\pi T} f(X_{\frac{i}{T} } )+
\sum_{i=1}^T  \tfrac{N-N'}{2\pi T N' }f(X_{\frac{i}{T} } )+
N \sum_{i=1}^T \left(
|\D^{i}_{N'}|-\tfrac{1}{2\pi T N' }\right) f(X_{\frac{i}{T} } ) \\
&\hspace{8cm} +N
\sum_{i=1}^T  |{\D}^{i}_{N'}|\omega_f(\|X\|_{\mathcal{C}^\alpha} T^{-\alpha} ).
\end{align*}
Comparing the Riemann sum with the integral and $f$ to its upper bound, we turn this inequality into
\begin{align*}
2\pi\sum_{i=1}^T N f({\D}^{i}_{N'})
&\leq\int_0^1 f(X_u)\d u +\omega_f( \|X\|_{\mathcal{C}^\alpha}T^{-\alpha} )+\|f\|_{\infty} \tfrac{N-N'}{ N'}+
\|f\|_{\infty} N \sum_{i=1}^T \left(
2\pi|\D^{i}_{N'}|-\tfrac{1}{T N' }\right) \\
&\hspace{7.8cm}+2\pi\omega_f(\|X\|_{\mathcal{C}^\alpha} T^{-\alpha} )N
\sum_{i=1}^T  |{\D}^{i}_{N'}|.
\end{align*}
Our next goal is to bound the last three terms of the right-hand side. Let us discuss the first, then the third and finally the second.

For the first term, it follows from the definition of $N'$ and by elementary arguments that for $N$ large enough, indeed larger than a certain $N_1$ that does not depend on $f$,
\[\frac{N-N'}{N'}<2 (N^{m+t-1}+\gamma N^{-\frac{1}{\gamma}+1}).\]


For the third term, 
since the event $G_{N}$ holds, we have
\[ \sum_{i=1}^T |{\D}^{i}_{N'}| \leq T \max_{i\in \{1,\dots, T\} } |\D^{i}_{N'}|\leq
\frac{1}{\pi N} 
.\]

Finally, since the event $E_{N}$ holds, and for $N$ large enough,
\[\sum_{i=1}^T \left(2\pi
|\D^{i}_{N'}|-\tfrac{1}{ T N' }\right)\leq N'^{-1-\delta} T^{\frac{1}{2}+\frac{s}{t}}\leq 2 N^{-1-\delta+\frac{t}{2}+s }.\]
Here the second inequality holds for $N$ larger than a certain $N_2$ which does not depend on $f$.

We end up with
\begin{align}
2\pi \sum_{i=1}^T N f({\D}^{i}_{N'})-  \!\int_0^1\! f(X_u)\d u &
\leq 3 \omega_f(\|X\|_{\mathcal{C}^\alpha} T^{-\alpha}  )
+2 \|f\|_\infty ( N^{m+t-1} \!+\!\gamma N^{-\frac{1}{\gamma}+1 }\!+\!  N^{-\delta+\frac{t}{2}+s}).
\label{eq:bound2}
\end{align}

We now turn to the second term of the right-hand side of \eqref{eq:decompoSmall}. Since $F_N$ holds,
\begin{equation}
\label{eq:deux}
N\sum_{1\leq i<j\leq T} f(\D^{i,j}_M)\leq  N \|f\|_\infty  \sum_{1\leq i<j\leq T} |\D^{i,j}_M|\leq \|f\|_\infty N^{-\zeta}.
\end{equation}
Using \eqref{eq:decompoSmall}, \eqref{eq:bound2} and \eqref{eq:deux}, we get
that almost surely, for $N\geq \max(N_0,N_1,N_2)$,
\begin{equation}
\label{eq:bornesup}
2\pi N f(\mathcal{D}_N)- \! \int_0^1\! f(X_u)\d u
\leq 3 \omega_f(\|X\|_{\mathcal{C}^\alpha} T^{-\alpha}  )
+2 \|f\|_\infty ( N^{m+t-1} +\gamma N^{-\frac{1}{\gamma}+1}+ N^{-\delta+\frac{t}{2}+s}+N^{-\zeta}).
\end{equation}
To obtain this upper bound, we used the second inequality of Lemma \ref{le:decompo}, and the definition of $N'$ which was suggested by the term $N-T-M(T-1)$ that appears in it. A repetition of the exact same arguments, with the difference that $N'$ is now defined as the largest element of $\mathbb N^\gamma$ smaller than $N+T+M(T-1)$, and using the first inequality of Lemma \ref{le:decompo} instead of the second, yields the corresponding lower bound, saying that the left-hand side of \eqref{eq:bornesup} is larger than the opposite of the right-hand side of \eqref{eq:bornesup}. 


This concludes the proof when $f$ is non-negative. To remove this assumption, it suffices to decompose $f$ into the sum of its positive and negative parts.
%
%
%
%
%
%
%
%
%
%
%
%
\end{proof}

We now extend Lemma \ref{le:casXi} from $N\in \mathbb{N}^\gamma$ to $N\in \mathbb{N}^*$, in order to obtain Lemma \ref{le:main}.

\begin{proof}[Proof of Lemma {\ref{le:main}}]
The reals $t$ and $\alpha$ being given, choose positive real numbers $s,\zeta,m,\delta, \gamma$ which satisfy the assumptions \eqref{eq:hyp}.
Set $\eta=\min(1- m-t,\frac{1}{\gamma}-1, \delta-\frac{t}{2}-s, \zeta)>0$.

Let us first assume $f$ is non-negative. Set $\tilde{N}=\max \{ n\in \mathbb{N}^\gamma: n\leq N\}$, the largest integer smaller than $N$ in $\mathbb{N}^\gamma$.

Since the sequence $(f(\mathcal{D}_N))_{N\geq 1}$ is non-increasing, we have
\begin{align*}
2\pi N f(\mathcal{D}_N) -\int_0^1 f(X_u) \d u
&\leq  2\pi N  f(\mathcal{D}_{\tilde{N}}) -\int_0^1 f(X_u) \d u\\
&=\frac{N}{\tilde{N}} \Big(2\pi \tilde{N}  f(\mathcal{D}_{\tilde{N}}) -\int_0^1 f(X_u) \d u   \Big)+ \Big(\frac{N}{\tilde{N}}-1\Big) \int_0^1 f(X_u) \d u.
%
\end{align*}
The first term is taken care of by Lemma \ref{le:casXi} and the fact that $N\leq 2\tilde N$ for $N$ large enough. The second term is bounded above, for $N$ sufficiently large, by $2\gamma \|f\|_\infty N^{-\frac{1}{\gamma}+1}$. Altogether, we find the upper bound
\[2\pi N f(\mathcal{D}_N) -\int_0^1 f(X_u) \d u \leq  C \big(\omega_f(\|X\|_{\mathcal{C}^\alpha} T^{-\alpha}) +\|f\|_\infty N^{-\eta} \big)\]
for some constant $C$. The corresponding lower bound is obtained by the same argument with $\tilde N$ defined as $\min\{n\in \mathbb N^\gamma : n\geq N\}$.
This concludes the proof when $f$ is non-negative. For the general case, we simply decompose $f$ into its positive and negative parts. This concludes the proof of Lemma~\ref{le:main}, and also the proof of Theorem~\ref{th:main}.
\end{proof}

\section{Further perspectives}

It is possible that a similar result also holds when we consider the joint windings of independent Brownian motions. To be more specific, for two independent planar Brownian motions $X,X'$, we can define their intersection measure $\ell$, which is carried by the plane (see \cite{rosen}).

One possible way to approximate the mass of this measure is to look at the Lebesgue measure of the intersection of Wiener sausages with small radius $\epsilon$ around $X$ and $X'$. In \cite{leGall} (and also in \cite{leGall2}), it is shown that $\ell(\mathbb{R}^2)$ can be obtained as the properly normalized limit of these measures as $\epsilon\to 0$.

For two independent planar Brownian motions $X,X'$, define
\[
\mathcal{D}^{(2)}_{N}=\{ z\in \R^2: \theta_X(z)\geq N, \theta_{X'}(z)\geq N\}.
\]
\begin{conjecture}
There exists a constant $C$ which depends only $\|X_0-X'_0\|$ and such that
 $CN^2 |\mathcal{D}^{(2)}_{N}|$ converges, as $N\to \infty$, towards $\ell(\mathbb{R}^2)$. The converges holds both in $L^p$ for any $p\in[1,+\infty)$ and almost surely.

Besides, almost surely, the measure $CN^2\mathbbm{1}_{\mathcal{D}^{(2)}_{N}}\d z$ converges weakly towards $\ell$.
\end{conjecture}
For such a result to hold, it is necessary that the exponent of $N$ is equal to $2$. Nonetheless, we cannot exclude that some logarithmic corrections should be added.



\bibliographystyle{plain}
\bibliography{bib2.bib}

\end{document}